\documentclass{birkjour}

\usepackage[utf8]{inputenc}
\usepackage[T1]{fontenc}
\usepackage{amsmath, amsfonts, amssymb, amsthm}
\usepackage[none]{hyphenat}
\usepackage{graphicx}
\usepackage{float}
\usepackage{times}
\usepackage{esint}
\usepackage{lipsum}
\usepackage{enumitem}
\usepackage{amsmath, amsthm, amscd, amsfonts, amssymb, graphicx, color}
\usepackage[bookmarksnumbered, colorlinks, plainpages]{hyperref}
\hypersetup{colorlinks=true,linkcolor=red, anchorcolor=green, citecolor=cyan, urlcolor=red, filecolor=magenta, pdftoolbar=true}
\usepackage{orcidlink}
\usepackage[scr]{rsfso}

\newtheorem{theorem}{Theorem}[section]
\newtheorem{lemma}[theorem]{Lemma}
\newtheorem{proposition}[theorem]{Proposition}

\theoremstyle{definition}

\theoremstyle{remark}
\newtheorem{remark}[theorem]{Remark}
\numberwithin{equation}{section}

\numberwithin{equation}{section}

\begin{document}

%
%
%
%
%
%
%
%
%

\title{Density properties for fractional Musielak-Sobolev spaces}

\author{EL-Houcine OUALI* \orcidlink{0009-0007-9106-6441}}
\address{{ Department of Mathematics and Computer Science},
	{A\"in Chock Faculty,  Hassan II University}, 
	{ B.P. 5366 Maarif, Casablanca},
	{Morocco}}
\email{oualihoucine4@gmail.com}

\author{Azeddine BAALAL}
\address{{ Department of Mathematics and Computer Science},
	{A\"in Chock Faculty,  Hassan II University}, 
	{ B.P. 5366 Maarif, Casablanca},
	{Morocco}}
\email{abaalal@gmail.com}
 \author{Mohamed BERGHOUT}
 	\address{{Laboratory of Partial Differential Equations, Algebra and Spectral Geometry, Department of Mathematics, Ibn Tofail University}, {P.O. Box 242-Kenitra 14000, Kenitra}, {Morocco}}
\email{Mohamed.berghout@uit.ac.ma; moh.berghout@gmail.com}


\date{}

\begin{abstract}
	  In this paper, we introduce a new fractional Musielak-Sobolev space $W^{s,\Phi_{x,y}}(\Omega)$ where $\Omega$ is an open subset in $\mathbb{R}^{N}$ and  we show some density properties of smooth and compactly supported functions in this space.
\end{abstract}

\maketitle

\section{Introduction}
In recent years there has been an increasing attention on problems involving the theory of fractional modular spaces, in particular the fractional order Orlicz-Sobolev spaces $W^{s,\Phi}(\Omega)$ (see \cite{BjSa2019,bot2020basic,bor2023anewclass,bsoh2020embedding}), and the fractional Sobolev spaces with variable exponents $W^{s,q(.),p(.,.)}(\Omega)$ (see \cite{barv2018onanew,babm2018density,km2023bourgan,dlrj2017traces,kjv2017fract,baalal2018traces}), which are two natural extensions of classical fractional Sobolev spaces $ W^{s,p}(\Omega)$ (see \cite{di2012hitchhikers,bms2019radial}),  and they are two special kinds of fractional Musielak-Sobolev spaces (see \cite{abss2022class, abss2023emb,aacs2023onfrac}). Particularly, one of this problems is the density of smooth and compactly supported functions in these spaces, see \cite{babm2018density,BjSa2019,ABS2019Intro,fiscella2015density}.

The objective of this paper is to provide some density properties of smooth and compactly supported functions in a new fractional Musielak-Sobolev space.

 We suppose that $\Omega$ is an open subset in $ \mathbb{R}^{N}$, $N\geqslant 1 $ and  $\Phi: \Omega \times \Omega \times[0, \infty) \rightarrow$ $[0, \infty)$ is a Carathéodory function defined by
\begin{equation}\label{eq1}
\Phi_{x,y}(t):=\Phi(x, y, t)=\int_0^t \varphi(x, y, s) d s,
\end{equation}
where $\varphi: \Omega \times \Omega \times[0, \infty) \rightarrow[0, \infty)$ is a function satisfying the following assumptions:
\begin{itemize}
\item[(i)] $\varphi(x, y, 0)=0$ and $\varphi(x, y, t)>0$ for any $(x, y) \in \Omega \times \Omega, t>0$;
\item[(ii)] $t\mapsto\varphi(x, y, t)$ is increasing on $[0, \infty)$;
\item[(iii)] $t\mapsto\varphi(x, y, t)$ is continuous on $[0, \infty)$, $\displaystyle\lim _{t \rightarrow 0} \varphi(x, y, t)=0$ and $\displaystyle\lim _{t \rightarrow \infty} \varphi(x, y, t)=\infty$, for all $(x,y)\in \Omega\times \Omega$.
\end{itemize}
It is not hard to show that hypotheses $(i)-(i i i)$ imply that $\Phi_{x,y}( t)$ is continuous, strictly increasing and convex in $t \geq 0$ for any $(x, y) \in \Omega \times \Omega$. Moreover, $\Phi_{x,y}(0)=0$, $\Phi_{x,y}(t)>0$ for all $t>0$ and $\Phi_{x,y}(t)$ has a sublinear decay at zero and a superlinear growth near infinity in the variable $t$ :
$$
\lim _{t \rightarrow 0^{+}} \frac{\Phi_{x,y}(t)}{t}=0, \quad \lim _{t \rightarrow \infty} \frac{\Phi_{x,y} (t)}{t}=\infty .
$$
A function $\Phi$ defined as before is known as generalized N-function.

We also consider the function $\widehat{\varphi}_x(t):=\widehat{\varphi}(x, t)=\varphi(x, x,t)$ for all $(x, t) \in \Omega \times [0,+\infty)$. This function is increasing homeomorphism from $\mathbb{R^{+}}$ onto itself. Let us set
\begin{equation}\label{eq2}
\widehat{\Phi}_x(t):=\widehat{\Phi}(x, t)=\int_0^t \widehat{\varphi}_x(\tau) \mathrm{d} \tau \text { for all } t \geqslant 0 .
\end{equation}
Then, $\widehat{\Phi}_x$ is also a generalized N-function.

\begin{remark}
 We know that for all $(x, y) \in \Omega \times \Omega$, $\Phi_{x, y}: \mathbb{R^{+}} \rightarrow \mathbb{R^{+}}$ is a strictly increasing continuous function. It follows that for each $(x, y) \in \Omega \times \Omega$, the function $t \mapsto$ $\Phi_{x, y}(t)$ has an inverse function denoted by $t \mapsto \Phi_{x, y}^{-1}(t)$. 
\end{remark}

For technical reasons, we will assume, throughout this paper, that :
\begin{itemize}
\item[(a)] $\delta K(x,y)\in L_{\Phi_{x,y}}(\Omega\times \Omega)$ where $\displaystyle K(x,y)=\frac{1}{|x-y|^s \Phi_{x,y}^{-1}\left(|x-y|^N\right)}$ and\\ $\displaystyle \delta=\min\left\lbrace 1,|x-y|\right\rbrace $.
\item[(b)] $\Phi_{x,y}$ and $\widehat{\Phi}_{x} $ are locally integrables, that is for any constant number $c > 0$ and for every compact set $A \subset \Omega$ we have : 

\begin{equation}\label{eq1.3}
\int_{A \times A}\Phi_{x,y}(c)dx dy < \infty  \textsl{ and } \int_{A}\widehat{\Phi}_{x}(c)dx < \infty .
\end{equation}
\item[(c)] $\Phi_{x,y}$ and $\widehat{\Phi}_{x} $ satisfy the following boundedness conditions : 
\begin{equation}\label{eq1.333}
\sup_{(x,y)\in \Omega\times \Omega}\Phi_{x,y}(1)<\infty \textsl{ and } \sup_{x\in\Omega}\widehat{\Phi}_{x}(1)<\infty.
\end{equation}

\item[(d)] 
\begin{equation}\label{eq1.4}
\Phi(x-z,y-z,t)=\Phi(x,y,t) \quad\quad \forall (x,y),(z,z)\in \Omega \times \Omega \textit{, } \forall t \geqslant 0.
\end{equation}

\item[(e)] There exist two positive constants $\varphi^{+}$and $\varphi^{-}$ such that
\begin{equation}\label{eq1.5}
1<\varphi^{-} \leqslant \frac{t \varphi_{x, y}(t)}{\Phi_{x, y}(t)} \leqslant \varphi^{+}<+\infty \quad \forall (x, y) \in \Omega \times \Omega, \quad \forall t \geqslant 0 .
\end{equation}
\end{itemize}
This relation implies that
\begin{equation}\label{eq1.6}
1<\varphi^{-} \leqslant \frac{t \widehat{\varphi}_x(t)}{\widehat{\Phi}_x(t)} \leqslant \varphi^{+}<+\infty, \quad \forall x \in \Omega,\quad \forall t \geqslant 0 .
\end{equation}

In light of assumption \eqref{eq1.5}, $\Phi_{x, y}$ and $\widehat{\Phi}_x$ satisfy the so-called $\Delta_2$-condition (see \cite{mmrv2008neu}), written $\Phi_{x, y} \in \Delta_2$ and $\widehat{\Phi}_x \in \Delta_2$, that is there exist two positive constants $C_1$ and $C_2$ such that
\begin{equation}
\Phi_{x, y}(2 t) \leqslant C_1 \Phi_{x, y}(t) \quad \text { for all }(x, y) \in \Omega \times \Omega, \quad \text { and all } t \geqslant 0 \text {, }
\end{equation}
and
\begin{equation}\label{1.6eq2}
\widehat{\Phi}_x(2 t) \leqslant C_2 \widehat{\Phi}_x(t) \quad \text { for all } x \in \Omega, \quad \text { and all } t \geqslant 0.
\end{equation}

For the function $\widehat{\Phi}_x$ given in \eqref{eq2}, the Musielak class is defined by
$$
K_{\widehat{\Phi}_x}(\Omega)=\left\{u: \Omega \longrightarrow \mathbb{R} \text { mesurable : } \int_{\Omega} \widehat{\Phi}_x(|u(x)|) \mathrm{d} x<\infty\right\} \text {, }
$$
and the Musielak space $L_{\widehat{\Phi}_x}(\Omega)$ is defined by
$$
L_{\widehat{\Phi}_x}(\Omega)=\left\{u: \Omega \longrightarrow \mathbb{R} \text { mesurable : } \int_{\Omega} \widehat{\Phi}_x(\lambda|u(x)|) \mathrm{d} x<\infty \text { for some } \lambda>0\right\} .
$$
The space $L_{\widehat{\Phi}_x}(\Omega)$ is a Banach space endowed with the Luxemburg norm
$$
\|u\|_{L_{\widehat{\Phi}_{x}}(\Omega)}=\inf \left\{\lambda>0: \int_{\Omega} \widehat{\Phi}_x\left(\frac{|u(x)|}{\lambda}\right) \mathrm{d} x \leqslant 1\right\} .
$$
The inequality \eqref{1.6eq2} implies that $ L_{\widehat{\Phi}_x}(\Omega)=K_{\widehat{\Phi}_x}(\Omega)$ (see \cite{musju1983Orlicz}).

Let $s\in (0,1)$, for the function $\Phi_{x,y}$ given in \eqref{eq1} we introduce the fractional Musielak-Sobolev space $W^{s,\Phi_{x, y}}(\Omega)$ as follows
$$
\begin{aligned}
&W^{s,\Phi_{x, y}}(\Omega)\\
&=\left\{u \in L_{\widehat{\Phi}_x}(\Omega): \int_{\Omega} \int_{\Omega} \Phi_{x, y}\left(\lambda|u(x)-u(y)|K(x,y)\right)\mathrm{d} x \mathrm{~d} y<\infty \textsl{ for some } \lambda >0 \right\}.
\end{aligned}
$$
This space is endowed with the norm
\begin{equation}\label{Norm}
\|u\|_{W^{s,\Phi_{x, y}}(\Omega)}:=\|u\|_{L_{\widehat{\Phi}_{x}}(\Omega)}+[u]_{s, \Phi_{x, y}},
\end{equation}
where $[.]_{s, \Phi_{x, y}}$ is the Gagliardo seminorm defined by
$$
[u]_{s, \Phi_{x, y}}=\inf \left\{\lambda>0: \int_{\Omega} \int_{\Omega} \Phi_{x, y}\left(\frac{|u(x)-u(y)|K(x,y)}{\lambda}\right)\mathrm{d} x \mathrm{~d} y \leqslant 1\right\} .
$$

\begin{remark}
\begin{itemize}

\item[(a)] The fractional  Musielak-Sobolev space was introduced in \cite{abss2022class} as follows 
$$
\begin{aligned}
&W^{s,\Phi_{x, y}}(\Omega)\\
&=\left\{u \in L_{\hat{\Phi}_x}(\Omega): \int_{\Omega} \int_{\Omega} \Phi_{x, y}\left(\frac{|u(x)-u(y)|}{\lambda|x-y|^{s}}\right)\frac{\mathrm{d} x \mathrm{~d} y}{|x-y|^{N}}<\infty \textsl{ for some  } \lambda>0 \right\}.
\end{aligned}
$$
We observe that this definition creates problems in the calculus and in the embedding results, for example, the Borel measure defined as $d \mu=\frac{d x d y}{|x-y|^N}$ is not finish in the neighbourhood of the origin.

\item[(b)] For the case: $\Phi_{x, y}(t)=\Phi(t)$, i.e., when $\Phi$ is independent of the variables $x$ and $y$, we have that $L_{\Phi}$ and $W^{s,\Phi}$ are Orlicz spaces and fractional Orlicz-Sobolev spaces respectively (see \cite{ABS2019Intro} ) such that
$$
\begin{aligned}
&W^{s,\Phi}(\Omega)\\
&=\left\{u \in L_{\Phi}(\Omega): \int_{\Omega} \int_{\Omega} \Phi\left(\frac{|u(x)-u(y)|}{\lambda|x-y|^s\Phi^{-1}\left(|x-y|^N \right) }\right) \mathrm{d} x \mathrm{~d} y<\infty \textsl{ for some }\lambda>0 \right\}.
\end{aligned}
$$
\item[(c)] For the case: $\Phi_{x, y}(t)=|t|^{p(x, y)}$ for all $(x, y) \in \Omega \times \Omega$, where $p:  \Omega \times \Omega \longrightarrow(1,+\infty)$ is a continuous bounded function such that
\begin{equation}
1<p^{-}=\min _{(x, y) \in  \Omega \times \Omega} p(x, y) \leqslant p(x, y) \leqslant p^{+}=\max _{(x, y) \in  \Omega \times \Omega} p(x, y)<+\infty,
\end{equation}
\end{itemize}
and $p$ is symmetric, that is,  $p(x, y)=p(y, x) \text { for all }(x, y) \in \Omega \times \Omega$.\\
If we set $\overline{p}(x)=p(x, x)$ for all $x \in \Omega$, then, we replace $L_{\widehat{\Phi}_x}$ by $L^{\overline{p}(x)}$, and $W^{s,\Phi_{x, y}}$ by $W^{s, p(x, y)}$ and we refer them as variable exponent Lebesgue spaces, and fractional Sobolev spaces with variable exponent respectively, (see \cite{baalal2018traces,barv2018onanew,babm2018density,abs2019eigenvalue,dlrj2017traces,kjv2017fract}), defined by
$$
L^{\overline{p}(x)}(\Omega)=\left\{u: \Omega \longrightarrow \mathbb{R} \text { measurable }: \int_{\Omega}|u(x)|^{\overline{p}(x)} \mathrm{d} x<+\infty\right\},
$$
and
$$
\begin{aligned}
&W^{s, p(x, y)}(\Omega)\\
&=\left\{u \in L^{\overline{p}(x)}(\Omega): \int_{\Omega \times \Omega} \frac{|u(x)-u(y)|^{p(x, y)}}{\lambda^{p(x, y)}|x-y|^{s p(x, y)+N}} \mathrm{~d} x \mathrm{~d} y<+\infty \text {, for some } \lambda>0\right\}.
\end{aligned}
$$
with the norm
$$
\|u\|_{W^{s, p(x, y)}(\Omega)}=\|u\|_{L^{\overline{p}(x)}(\Omega)}+[u]_{W^{s, p(x, y)}(\Omega)},
$$
where $[.]_{W^{s, p(x, y)}(\Omega)}$ is the Gagliardo seminorm with variable exponent given by
$$
[u]_{W^{s, p(x, y)}(\Omega)}=\inf \left\{\lambda>0: \int_{\Omega \times \Omega} \frac{|u(x)-u(y)|^{p(x, y)}}{\lambda^{p(x, y)}|x-y|^{N+s p(x, y)}} \mathrm{d} x \mathrm{~d} y \leqslant 1\right\} .
$$
\end{remark}
By the same way in \cite[Theorem 2.1]{abss2022class} it is easy to prove that $ W^{s,\Phi_{x, y}}(\Omega) $ is a separable and reflexive Banach space.

Let $s\in (0,1)$, we define the space $W^{s}K_{\Phi_{x,y}}(\Omega)$ as follows 

$$  W^s K_{\Phi_{x, y}}(\Omega)=\left\{u \in K_{\widehat{\Phi}_x}(\Omega): |u(x)-u(y)|K(x,y)\in K_{\Phi_{x,y}}(\Omega \times \Omega) \right\}.  $$ 

\begin{remark}
\begin{itemize}
\item[(a)] We have $W^s K_{\Phi_{x, y}}(\Omega) \subseteq W^{s,\Phi_{x, y}}(\Omega)$.

\item[(b)] $W^s K_{\Phi_{x, y}}(\Omega)= W^{s,\Phi_{x, y}}(\Omega)$ if and only if $ \Phi_{x,y}$ satisfies the $\Delta_{2}$-condition.
\end{itemize}
\end{remark}

Let $C(\overline{\Omega}) $ denote the space of uniformly continuous functions equipped with the supremum norm $\|u\|_{\infty}=\sup_{x\in \overline{\Omega}}|u(x)|$. By $C^k(\overline{\Omega}), k \in \mathbb{N}$, we denote the space of all functions $u$ such that $\partial_\alpha u:=\frac{\partial^{|\alpha|} u}{\partial^{\alpha_1} x_1 \cdots \partial^{\alpha_n} x_n} \in C(\overline{\Omega})$ for all multi-indices $\alpha=\left(\alpha_1, \alpha_2, \ldots, \alpha_n\right),|\alpha|:=\alpha_1+\alpha_2+\cdots+\alpha_n \leq k$. This space is equipped with the norm $\sup _{|\alpha| \leq k}\left\|\partial_\alpha f\right\|_{\infty}$. Let $C^{\infty}(\overline{\Omega})=\bigcap^{\infty}_{k=0} C^k(\overline{\Omega})$ and let  $C^{\infty}(\Omega)$ denote the set of all infinitely differentiable functions in $\Omega$. The subspace $C_0^{\infty}(\Omega)$ consist of all those functions in $C^{\infty}(\Omega)$ that have compact support in $\Omega$. We also denote by C a constant whose value may change from line to line.

We say that an open domain $\Omega \subset \mathbb{R}^N$ is a $W^{s,\Phi_{x,y}}$-extension domain if there exists a continuous linear extension operator
$$
\mathcal{E}:W^{s,\Phi_{x,y}}(\Omega) \longrightarrow W^{s,\Phi_{x,y}}\left(\mathbb{R}^N\right)
$$
such that $\left.\mathcal{E} u\right|_{\Omega}=u$ for each $u \in W^{s,\Phi_{x,y}}(\Omega)$. For example domains with  Lipschitz boundary are $W^{s,\Phi_{x,y}}$-extension domains (the proof is similar to the one of Theorem 2.4 in \cite{abss2023emb}).

This paper is organized as follows.  In section \ref{sec2}, we state technical and
elementary lemmas useful along the paper. The proof of the main results is given in Section \ref{sec3}.
The first main result is an approximation with a continuous and compactly
 supported function, see Theorem \ref{theo3.1}. The second main result is an approximation with smooth and compactly supported
 functions, see Theorem \ref{theo3.2}. The third result of our paper is the density of $C^{\infty}(\overline{\Omega})$ in $W^{s,\Phi_{x,y}}(\Omega)$, where $\Omega$ is a $W^{s,\Phi_{x,y}}$-extension domain, see Theorem \ref{theo3.3}.

\section{ Some preliminaries results}\label{sec2}

This section is devoted to the proof of some preliminary lemmas, which will be used throughout this paper.

From \cite[Lemma 2.2]{abss2022class} we have the following lemma.
\begin{lemma}\label{lemmapro}
Assume that \eqref{eq1.5} is satisfied. Then the following inequalities hold true
\begin{equation}
\Phi_{x, y}(\sigma t) \geqslant \sigma^{\varphi^{-}} \Phi_{x, y}(t) \quad \quad \forall t>0,\quad \forall\sigma>1,
\end{equation}
\begin{equation}
\Phi_{x, y}(\sigma t) \geqslant \sigma^{\varphi^{+}} \Phi_{x, y}(t),\quad \forall t>0, \quad \forall\sigma \in(0,1),
\end{equation}
\begin{equation}
\Phi_{x, y}(\sigma t) \leqslant \sigma^{\varphi^{+}} \Phi_{x, y}(t) \quad \forall t>0,\quad \forall\sigma>1,
\end{equation}
\begin{equation}
\Phi_{x, y}(t) \leqslant \sigma^{\varphi^{-}} \Phi_{x, y}\left(\frac{t}{\sigma}\right)\quad \forall t>0,\quad \forall \sigma \in(0,1).
\end{equation}
\end{lemma}
From  \cite[Theorem~3.35]{vigelis2011musielak}  we obtain the following proposition.

\begin{proposition}\label{2prop1} Let $u,u_{n}\in L_{\widehat{\Phi}_{x}}(\Omega)$, $n\in \mathbb{N}$. Then the following statements are equivalent
\begin{itemize}
\item[(i)] $\displaystyle\lim_{n\longrightarrow +\infty} \|u_{n}-u\|_{L_{\widehat{\Phi}_{x}}(\Omega)}=0.$

\item[(ii)] $\displaystyle \lim_{n\longrightarrow +\infty} \int_{\Omega}\widehat{\Phi}_{x}(|u_{n}-u|)dx=0.$
\end{itemize} 

\end{proposition}
\begin{proposition}\cite[Remark~2.1]{ayya2020someapp}\label{2prop2} Let $G$ be a generalized N-function satisfying \eqref{eq1.3} and the $\Delta_{2}$-condition. Then $C^{\infty}_{0}(\Omega)$ is dense in $ (L_{G}(\Omega), \|.\|_{L_{G}(\Omega)})$.

\end{proposition}
 In the sequel, let $\mathcal{B}_c(\Omega)$ be the set of bounded functions compactly supported in $\Omega$, let $B_{R}$ be the ball with the center $0$ and the radius $R>0$ and let $|B_{R}|$ be the Lebesgue measure
 of the ball $B_{R}$.

The proofs of Theorem \ref{theo3.2} and Theorem \ref{theo3.3} are mainly based on a basic technique of convolution (which makes functions infinitely differentiable), joined with a cut-off (which makes their support compact). In the remainder of this article, we describe properties of these operations with respect to the norm in fractional Musielak-Sobolev spaces.

Let $J$ stand for the Friedrichs mollifier kernel defined on $\mathbb{R}^N$ by
$$
J(x)= \begin{cases}k \mathrm{e}^{-1 /\left(1-|x|^2\right)} & \text { if }|x|<1, \\ 0 & \text { if } |x| \geqslant 1,\end{cases}
$$
where $k>0$ is such that $\displaystyle\int_{\mathbb{R}^{N}} J(x) \mathrm{d} x=\int_{B_{1}} J(x) \mathrm{d} x=1$.\\
 For $\varepsilon>0$, we define $J_{\varepsilon}(x)=\varepsilon^{-N} J\left(x \varepsilon^{-1}\right)$ and for any $u \in W^{s,\Phi_{x,y}}\left(\mathbb{R}^N\right)$, let us denote by $u_{\varepsilon}$ the function defined as the convolution between $u$ and $J_{\varepsilon}$ ; that is,
$$
u_{\varepsilon}(x)=(u*J_{\varepsilon})(x)=\int_{\mathbb{R}^N} J_{\varepsilon}(x-y) u(y) \mathrm{d} y=\int_{B_{1}} u(x-\varepsilon y) J(y) \mathrm{d} y.
$$
Note that, by construction, we have $u_{\varepsilon} \in C^{\infty}\left(\mathbb{R}^N\right)$. 

\begin{lemma}\cite[ Corollary 2.1]{ayya2020someapp}\label{lemmacorollary}  Let $G$ be a generalized N-function satisfying \eqref{eq1.3} and let $u \in \mathcal{B}_c(\Omega)$. For any $\varepsilon>0$ small enough, we have $u_{\varepsilon} \in C_0^{\infty}(\Omega)$. Furthermore,
$$
\left\|u_{\varepsilon}-u\right\|_{L_G(\Omega)} \rightarrow 0 \quad \text { as } \varepsilon \rightarrow 0 .
$$
\end{lemma}

\begin{lemma}\cite[ Lemma B.1]{ayya2020someapp}\label{lemma22}
Let $G$ be a generalized N-function satisfying \eqref{eq1.3} and the $\Delta_{2}$-condition. Then
$\mathcal{B}_c(\Omega)$ is dense in $(L_G(\Omega), \|u\|_{L_G(\Omega)})$.
\end{lemma}

\begin{lemma}\label{2lem1} Let $u \in C_0^{\infty}\left(\mathbb{R}^N\right)$. Then there exists a positive constant $C$ such that
$$
\int_{\mathbb{R}^N} \int_{\mathbb{R}^N} \Phi_{x,y}\left( |u(x)-u(y)|K(x,y)\right)d x d y\leq C .
$$
\end{lemma}
\begin{proof} Let $u \in C_0^{\infty}\left(\mathbb{R}^N\right)$ with $\operatorname{supp} u \subseteq B_R$. Therefore, $u$ vanishes outside $B_R$. Thus 
$$
\begin{aligned}
 &\int_{\mathbb{R}^N} \int_{\mathbb{R}^N} \Phi_{x,y}\left(|u(x)-u(y)|K(x,y)\right)d x d y \\
& =\int_{B_R} \int_{B_R} \Phi_{x,y}\left(|u(x)-u(y)|K(x,y)\right)d x d y\\
 &+2\int_{B_R} \int_{\mathbb{R}^{N}\backslash B_{R}} \Phi_{x,y}\left(|u(x)-u(y)|K(x,y)\right)d x d y \\
&\leq 2\int_{B_R} \int_{\mathbb{R}^{N}} \Phi_{x,y}\left(|u(x)-u(y)|K(x,y)\right)d x d y
\end{aligned}
$$
Now, we have
$$
|u(x)-u(y)| \leqslant|| \nabla u \|_{L^{\infty}\left(\mathbb{R}^N\right)}|x-y| \text { and }|u(x)-u(y)| \leqslant 2\|u\|_{L^{\infty}\left(\mathbb{R}^N\right)} .
$$
Then
$$
|u(x)-u(y)| \leqslant 2\|u\|_{C^1\left(\mathbb{R}^N\right)} \min \{1,|x-y|\}:=\alpha \delta(x, y).
$$
with $\alpha=2\|u\|_{C^1\left(\mathbb{R}^N\right)}$. Since $\delta K \in L_{\Phi_{x,y}}\left(\mathbb{R}^N \times \mathbb{R}^N\right)$, then
$$
\begin{aligned}
&\int_{\mathbb{R}^N} \int_{\mathbb{R}^N} \Phi_{x,y}\left(|u(x)-u(y)|K(x,y)\right)d x d y\\
&\leq \int_{\mathbb{R}^N} \int_{\mathbb{R}^N} \Phi_{x,y}\left(\alpha \delta(x, y)K(x,y)\right)d x d y\\
&\leq C \int_{\mathbb{R}^N} \int_{\mathbb{R}^N} \Phi_{x,y}\left( \delta(x, y)K(x,y)\right)d x d y < \infty.
\end{aligned}
$$
where the constant $C$ depends on $\alpha$, $\varphi^{+}$ and $\varphi^{-}$. The lemma is proved.
\end{proof}
Note that by Proposition \ref{2prop2} and Lemma \ref{2lem1}, we have $C^{\infty}_{0}(\mathbb{R}^{N})\subseteq W^{s,\Phi_{x,y}}(\mathbb{R}^{N}) $.
\begin{lemma}\label{2lem2}
 Let $u \in L_{\widehat{\Phi}_{x}}\left(\mathbb{R}^N\right)$. Then there exists a sequence of functions $u_m \in$ $L_{\widehat{\Phi}_{x}}\left(\mathbb{R}^N\right) \cap L^{\infty}\left(\mathbb{R}^N\right)$ such that
$$
\left\|u-u_m\right\|_{L_{\widehat{\Phi}_{x}}\left(\mathbb{R}^N\right)} \longrightarrow 0 \text { as } m \longrightarrow+\infty .
$$
\end{lemma}
\begin{proof}
 We set
$$
u_m(x):= \begin{cases}m & \text { if } u(x) \geq m, \\ u(x) & \text { if } u(x) \in(-m, m), \\ -m & \text { if } u(x) \leq-m .\end{cases}
$$
We have
$$
u_m \longrightarrow u \text { a.e. in } \mathbb{R}^N
$$
and
$$
\widehat{\Phi}_{x}\left( \left|u_m(x)\right|\right)  \leq \widehat{\Phi}_{x}\left( |u(x)|\right) \in L^1\left(\mathbb{R}^N\right) .
$$
Then by the dominated convergence theorem and Proposition \ref{2prop1}, the proof follows.
\end{proof}
\begin{lemma}\label{2lem3}
 Let $u \in L_{\Phi_{x,y}}\left(\mathbb{R}^N \times \mathbb{R}^N\right)$. Then there exists a sequence of functions $u_M \in L_{\Phi_{x,y}}\left(\mathbb{R}^N \times \mathbb{R}^N\right) \cap L^{\infty}\left(\mathbb{R}^N \times \mathbb{R}^N\right)$ such that
$$
\left\|u-u_M\right\|_{L_{\Phi_{x,y}}\left(\mathbb{R}^N \times \mathbb{R}^N\right)} \longrightarrow 0 \quad \text { as } M \longrightarrow+\infty \text {. }
$$
\end{lemma}
\begin{proof}
 We set
$$
u_M(x, y):= \begin{cases}M & \text { if } u(x, y) \geq M \\ u(x, y) & \text { if } u(x, y) \in(-M, M), \\ -M & \text { if } u(x, y) \leq-M.\end{cases}
$$
We have
$$
u_M \longrightarrow u \text { a.e. in } \mathbb{R}^N \times \mathbb{R}^N
$$
and
$$
\Phi_{x,y}\left(\left|u_M(x, y)\right|\right) \leq \Phi_{x,y}\left(|u(x, y)|\right) \in L^1\left(\mathbb{R}^N \times \mathbb{R}^N\right) .
$$
Then the result follows from the dominated convergence theorem and Proposition \ref{2prop1}.
\end{proof}
For small $\varepsilon$, convolutions do not change too much the norm \eqref{Norm}, according to the following result :
\begin{lemma}\label{lem2.06} Let $u \in W^{s,\Phi_{x,y}}\left(\mathbb{R}^N\right)$. We suppose that \eqref{eq1.3}, \eqref{eq1.4} and \eqref{eq1.5} hold. Then $\left\|u-u_{\varepsilon}\right\|_{W^{s,\Phi_{x,y}}\left(\mathbb{R}^N\right)} \rightarrow 0$ as $\varepsilon \rightarrow 0$.
\end{lemma}
\begin{proof}
 Let $u \in W^{s,\Phi_{x,y}}\left(\mathbb{R}^N\right)$. Since $u \in L_{\widehat{\Phi}_{x}}(\mathbb{R}^{N})$,
then according to Lemma \ref{lemma22}, we can assume that $u$ is bounded and compactly supported in $\mathbb{R}^{N}$. So that by Lemma \ref{lemmacorollary}, we have
\begin{equation}\label{eq2.01}
\left\|u-u_{\varepsilon}\right\|_{L_{\widehat{\Phi}_{x}}(\mathbb{R}^{N})} \rightarrow 0 \quad \text { as } \varepsilon \rightarrow 0 .
\end{equation}
Hence, by Proposition \ref{2prop1}, it suffices to prove that

\begin{equation}\label{eq2.1}
\int_{\mathbb{R}^{N} \times \mathbb{R}^{N}}\Phi_{x,y}\left(|u_\varepsilon(x)-u(x)-u_\varepsilon(y)+u(y)|K(x,y)\right)  d x d y \longrightarrow 0,
\end{equation}
as $\varepsilon \longrightarrow 0$. Using the definition of $u_{\varepsilon}$, the Jensen’s inequality, Tonelli's and Fubini's theorems, we get
\begin{equation}\label{eq2.007}
\begin{aligned}
& \int_{\mathbb{R}^{N} \times \mathbb{R}^{N}}\Phi_{x,y}\left(|u_\varepsilon(x)-u(x)-u_\varepsilon(y)+u(y)|K(x,y)\right)  d x d y  \\
& =\int_{\mathbb{R}^{N} \times \mathbb{R}^{N}} \Phi_{x,y}\left(\left[\int_{\mathbb{R}^N}(u(x-z)-u(y-z)) J_\varepsilon(z) d z-u(x)+u(y)\right]K(x,y)\right)d x d y \\
& =\int_{\mathbb{R}^{N} \times \mathbb{R}^{N}} \Phi_{x,y}\left(\left[\int_{B_1}(u(x-\varepsilon z)-u(y-\varepsilon z)-u(x)+u(y))K(x,y) J(z) d z\right]\right)d x d y\\
&\leq \int_{\mathbb{R}^{N} \times \mathbb{R}^{N}}\left[\int_{B_1}\Phi_{x,y}\left( |u(x-\varepsilon z)-u(y-\varepsilon z)-u(x)+u(y)|K(x,y) J(z)\right) d z\right] d x d y \\
& \leq \int_{\mathbb{R}^{N} \times \mathbb{R}^{N} \times B_1}\Phi_{x,y}\left( |u(x-\varepsilon z)-u(y-\varepsilon z)-u(x)+u(y)|K(x,y)\right)\\
&\times \left(J(z)^{\varphi^{+}}+J(z)^{\varphi^{-}}\right) d x d y d z .
\end{aligned}
\end{equation}
 Now, let us show that
\begin{equation} \label{eqclaim}
\int_{\mathbb{R}^{N} \times \mathbb{R}^{N}}\Phi_{x,y}\left(|u(x-\varepsilon z)-u(y-\varepsilon z)-u(x)+u(y)| K(x, y)\right) d x d y\longrightarrow 0,
\end{equation}
as $\varepsilon \longrightarrow 0$.\\
Fix $z \in B_1$ and put $w=(z, z) \in \mathbb{R}^{N} \times \mathbb{R}^{N}$. We define the function $v: \mathbb{R}^{N} \times \mathbb{R}^{N} \rightarrow \mathbb{R}$ by
$$
v(x, y)=(u(x)-u(y))K(x, y), \quad \forall(x, y) \in \mathbb{R}^{N} \times \mathbb{R}^{N} .
$$
Then $v \in L_{\Phi_{x,y}}(\mathbb{R}^{N} \times \mathbb{R}^{N})$. If $\varepsilon^{\prime}>0$, by proposition \ref{2prop2}, there exists $\mathrm{g} \in C_0^{\infty}(\mathbb{R}^{N} \times \mathbb{R}^{N})$ with $\displaystyle\|v-g\|_{L_{\Phi_{x,y}}(\mathbb{R}^{N} \times \mathbb{R}^{N})}<\frac{\varepsilon^{\prime}}{3}$, so
$$
\begin{aligned}
& \|v(.-\varepsilon w)-v\|_{L_{\Phi_{x,y}}(\mathbb{R}^{N} \times \mathbb{R}^{N})} \\
\leq & \|v(.-\varepsilon w)-g(.-\varepsilon w)\|_{L_{\Phi_{x,y}}(\mathbb{R}^{N} \times \mathbb{R}^{N})}+\|g(.-\varepsilon w)-g\|_{L_{\Phi_{x,y}}(\mathbb{R}^{N} \times \mathbb{R}^{N})}\\
&+\|v-g\|_{L_{\Phi_{x,y}}(\mathbb{R}^{N} \times \mathbb{R}^{N})} \\
\leq & \frac{\varepsilon^{\prime}}{3}+\frac{\varepsilon^{\prime}}{3}+\frac{\varepsilon^{\prime}}{3}=\varepsilon^{\prime},
\end{aligned}
$$
with $\varepsilon$ is sufficiently small. Then \eqref{eqclaim} follows.
Moreover, for a.e. $z \in B_1$, we have
\begin{equation}\label{eq2.3}
\begin{aligned}
& \left(J(z)^{\varphi^{+}}+J(z)^{\varphi^{-}}\right) \int_{\mathbb{R}^{N} \times \mathbb{R}^{N}}\Phi_{x,y}\left(|u(x-\varepsilon z)-u(y-\varepsilon z)-u(x)+u(y)|K(x,y)\right) d x d y \\
& \leq 2^{\varphi^{+}-1}\left(J(z)^{\varphi^{+}}+J(z)^{\varphi^{-}}\right)\left(\int_{\mathbb{R}^{N} \times \mathbb{R}^{N}}\Phi_{x,y}\left(|u(x-\varepsilon z)-u(y-\varepsilon z)|K(x,y)\right)d x d y\right. \\
& \left.+\int_{\mathbb{R}^{N} \times \mathbb{R}^{N}}\Phi_{x,y}\left(|u(x)-u(y)|K(x, y)\right) d x d y\right) \\
& \leq 2^{\varphi^{+}}\left(J(z)^{\varphi^{+}}+J(z)^{\varphi^{-}}\right) \int_{\mathbb{R}^{N} \times \mathbb{R}^{N}}\Phi_{x,y}\left(|u(x)-u(y)|K(x,y)\right) d x d y \in L^{\infty}\left(B_1\right),
\end{aligned}
\end{equation}
for any $\varepsilon>0$. Therefore, by using \eqref{eqclaim}, \eqref{eq2.3} and the dominated convergence theorem, we get
$$
\begin{aligned}
&\int_{B_1} \int_{\mathbb{R}^{N} \times \mathbb{R}^{N}}\Phi_{x,y}\left(|u(x-\varepsilon z)-u(y-\varepsilon z)-u(x)+u(y)|K(x, y)\right)\\&\left(J(z)^{\varphi^{+}}+J(z)^{\varphi^{-}}\right) d x d y d z \longrightarrow 0, \textsl{ as } \varepsilon\longrightarrow 0.
\end{aligned}
$$
From the above assertion and \eqref{eq2.007}, we obtain \eqref{eq2.1}. Using this fact in combination with \eqref{eq2.01}, we conclude our proof.
\end{proof}

Now, we will discuss the cut-off technique needed for the density argument. For any $j \in \mathbb{N}$, let $\tau_j \in C^{\infty}\left(\mathbb{R}^N\right)$ be such that
\begin{equation}\label{eqcut}
\begin{gathered}
0 \leq \tau_j(x) \leq 1, \quad \forall x \in \mathbb{R}^N, \\
\tau_j(x)= \begin{cases}1 & \text { if } x \in B_j, \\
0 & \text { if } x \in \mathbb{R}^N \backslash B_{j+1},\end{cases}
\end{gathered}
\end{equation}
where $B_j$ denotes the ball centered at zero with radius $j$.

\begin{lemma}\label{2lem4} Let $u \in W^{s,\Phi_{x,y}}\left(\mathbb{R}^N\right)$. Then $\tau_j u \in W^{s,\Phi_{x,y}}\left(\mathbb{R}^N\right)$.
\end{lemma}
\begin{proof}
Let $u \in W^{s,\Phi_{x,y}}\left(\mathbb{R}^N\right)$. It is clear that $\tau_j u \in L_{\widehat{\Phi}_{x}}\left(\mathbb{R}^N\right)$ since $\left|\tau_j\right| \leq 1$. Furthermore,
$$
\begin{aligned}
& \int_{\mathbb{R}^N} \int_{\mathbb{R}^N} \Phi_{x,y}\left(\left|\tau_j(x) u(x)-\tau_j(y) u(y)\right|K(x,y)\right) d x d y \\
&\leq  2^{\varphi^{+}-1} \int_{\mathbb{R}^N} \int_{\mathbb{R}^N} \Phi_{x,y}\left(\left|\tau_j(x)(u(x)-u(y))\right|K(x,y)\right)  d x d y \\
& +2^{\varphi^{+}-1} \int_{\mathbb{R}^N} \int_{\mathbb{R}^N} \Phi_{x,y}\left(\left|u(y)\left(\tau_j(x)-\tau_j(y)\right)\right|K(x,y)\right)d x d y\\
&\leq  2^{\varphi^{+}-1} \int_{\mathbb{R}^N} \int_{\mathbb{R}^N} \Phi_{x,y}\left(|u(x)-u(y)|K(x,y)\right)d x d y\\
& +2^{\varphi^{+}-1} \int_{\mathbb{R}^N} \int_{\mathbb{R}^N} \Phi_{x,y}\left(\left|u(y)\left(\tau_j(x)-\tau_j(y)\right)\right|K(x,y)\right)d x d y,
\end{aligned}
$$
where the integral
$$
\int_{\mathbb{R}^N} \int_{\mathbb{R}^N} \Phi_{x,y}\left(|u(x)-u(y)|K(x,y)\right)d x d y
$$
is finite since $u \in W^{s,\Phi_{x,y}}\left(\mathbb{R}^N\right)$.
According to Lemma \ref{2lem2}, we can assume that $u \in L^{\infty}\left(\mathbb{R}^N\right)$. Therefore,
\begin{equation}\label{ieqnew}
\begin{aligned}
&\int_{\mathbb{R}^N} \int_{\mathbb{R}^N} \Phi_{x,y}\left(\left|u(y)\left(\tau_j(x)-\tau_j(y)\right)\right|K(x,y)\right)dx dy\\
&\leq C \int_{\mathbb{R}^N} \int_{\mathbb{R}^N} \Phi_{x,y}\left(\left|\left(\tau_j(x)-\tau_j(y)\right)\right|K(x,y)\right)d x d y,
\end{aligned}
\end{equation}
where the constant $C$ depends on $\varphi^{+}, \varphi^{-}$, and $\|u\|_{L^{\infty}\left(\mathbb{R}^N\right)}$. Finally, since $ \tau_{j}\in C^{\infty}_{0}(\mathbb{R}^{N})$, then by Lemma \ref{2lem1} and the inequality \eqref{ieqnew}, we get
$$
\int_{\mathbb{R}^N} \int_{\mathbb{R}^N} \Phi_{x,y}\left(\left|u(y)\left(\tau_j(x)-\tau_j(y)\right)\right|K(x,y)\right) d x d y<\infty .
$$
This concludes the proof.
\end{proof}

\begin{lemma}\label{lem2.8}
 Let $u \in W^{s,\Phi_{x,y}}\left(\mathbb{R}^N\right)$. Then $\operatorname{supp}\left(\tau_j u\right) \subseteq \overline{B}_{j+1} \cap \operatorname{supp} u$, and
$$
\left\|\tau_j u-u\right\|_{W^{s,\Phi_{x,y}}\left(\mathbb{R}^N\right)} \longrightarrow 0 \quad \text { as } j \longrightarrow+\infty .
$$
\end{lemma}
\begin{proof}
By \eqref{eqcut} and \cite[Lemma~9 ]{fiscella2015density}, we get
$$
\operatorname{supp}\left(\tau_j u\right) \subseteq \overline{B}_{j+1} \cap \operatorname{supp} u .
$$
Now, let us show that
$$
\left\|\tau_j u-u\right\|_{W^{s,\Phi_{x,y}}\left(\mathbb{R}^N\right)} \longrightarrow 0 \text { as } j \longrightarrow+\infty .
$$
From Proposition \ref{2prop1}, it suffices to prove that
$$
\int_{\mathbb{R}^N}\widehat{\Phi}_{x}\left(\left|\tau_j(x) u(x)-u(x)\right|\right) d x \longrightarrow 0 \quad \text { as } j \longrightarrow+\infty
$$
and
$$
\int_{\mathbb{R}^N} \int_{\mathbb{R}^N} \Phi_{x,y}\left(\left|\tau_j(x) u(x)-u(x)-\tau_j(y) u(y)+u(y)\right|K(x,y)\right)d x d y\longrightarrow 0,
$$
as  $j \longrightarrow+\infty$. We observe that
$$
\begin{aligned}
\widehat{\Phi}_{x}\left(\left|\tau_j(x) u(x)-u(x)\right|\right) & \leq\widehat{\Phi}_{x}\left(2|u(x)|\right) \\
& \leq 2^{\varphi^{+}}\widehat{\Phi}_{x}\left(|u(x)|\right)\in L^1\left(\mathbb{R}^N\right) .
\end{aligned}
$$
Moreover, by \eqref{eqcut} we have
$$
\widehat{\Phi}_{x}\left(\left|\tau_j(x) u(x)-u(x)\right|\right) \longrightarrow 0 \quad \text { as } j \longrightarrow+\infty \text { a.e. in } \mathbb{R}^N.
$$
Then, by using the dominated convergence theorem, we get
$$
\int_{\mathbb{R}^N}\widehat{\Phi}_{x}\left(\left|\tau_j(x) u(x)-u(x)\right|\right) \longrightarrow 0 \quad \text { as } j \longrightarrow+\infty .
$$
Now, let us show that
$$
\int_{\mathbb{R}^N} \int_{\mathbb{R}^N} \Phi_{x,y}\left(\left|\tau_j(x) u(x)-u(x)-\tau_j(y) u(y)+u(y)\right|K(x,y)\right)d x d y \longrightarrow 0,
$$
as $j \longrightarrow+\infty$. We set $\eta_j=1-\tau_j$. Then $\eta_j u=u-\tau_j u$. Moreover,
$$
\begin{aligned}
& \left|\tau_j(x) u(x)-u(x)-\tau_j(y) u(y)+u(y)\right| \\
& \quad=\left|\eta_j(x)(u(x)-u(y))-\left(\tau_j(y)-\tau_j(x)\right) u(y)\right| .
\end{aligned}
$$
Therefore,
$$
\begin{aligned}
& \int_{\mathbb{R}^N} \int_{\mathbb{R}^N} \Phi_{x,y}\left(\left|\tau_j(x) u(x)-u(x)-\tau_j(y) u(y)+u(y)\right|K(x,y)\right)d x d y \\
& \leq 2^{\varphi^{+}-1} \int_{\mathbb{R}^N} \int_{\mathbb{R}^N} \Phi_{x,y}\left(\left|\tau_j(x)-\tau_j(y)\right|K(x,y)|u(y)|\right)d x d y \\
& \quad+2^{\varphi^{+}-1} \int_{\mathbb{R}^N} \int_{\mathbb{R}^N} \Phi_{x,y}\left(|u(x)-u(y)|K(x,y) \eta_j(x)\right) d x d y .
\end{aligned}
$$
Note that by Lemma \ref{2lem2}, we can assume that $u \in L^{\infty}\left(\mathbb{R}^N\right)$. Hence,
$$
\begin{aligned}
&\Phi_{x,y}\left(\left|\tau_j(x)-\tau_j(y)\right|K(x,y)|u(y)|\right)\\ &\leq C\left(\|u\|_{L^{\infty}\left(\mathbb{R}^N\right)}, \varphi^{+}, \varphi^{-}\right) \Phi_{x,y}\left(\left|\tau_j(x)-\tau_j(y)\right|K(x,y)\right).
\end{aligned}
$$
By Lemma \ref{2lem1}, we have
$$
\Phi_{x,y}\left(\left|\tau_j(x)-\tau_j(y)\right|K(x,y)\right) \in L^1\left(\mathbb{R}^N \times \mathbb{R}^N\right) .
$$
Furthermore, we have
$$
\Phi_{x,y}\left(\left|\tau_j(x)-\tau_j(y)\right|K(x,y)|u(y)|\right)  \longrightarrow 0 \quad \text { as } j \longrightarrow \infty \text { a.e. in } \mathbb{R}^N \times \mathbb{R}^N.
$$
Hence, by using the dominated convergence theorem, we get
$$
\int_{\mathbb{R}^N} \int_{\mathbb{R}^N} \Phi_{x,y}\left(\left|\tau_j(x)-\tau_j(y)\right||u(y)|K(x,y)\right)d x d y \longrightarrow 0 \quad \text { as } j \longrightarrow \infty .
$$
Also, we have
$$
\Phi_{x,y}\left(|u(x)-u(y)|\eta_j(x)K(x,y)\right)\ \leq \eta_j(x)^{\varphi^{-}} \Phi_{x,y}\left(|u(x)-u(y)|K(x,y)\right)
$$
and
$$
\Phi_{x,y}\left(|u(x)-u(y)|K(x,y)\right) \in L^1\left(\mathbb{R}^N \times \mathbb{R}^N\right).
$$
Again by \eqref{eqcut}, we have
$$
\Phi_{x,y}\left(|u(x)-u(y)|K(x,y)\eta_j(x)\right) \longrightarrow 0 \quad \text { as } j \longrightarrow \infty \text { a.e. in } \mathbb{R}^N \times \mathbb{R}^N .
$$
Hence, by the dominated convergence theorem, we have
$$
\int_{\mathbb{R}^N} \int_{\mathbb{R}^N} \Phi_{x,y}\left(|u(x)-u(y)|K(x,y)\eta_j(x)\right)d x d y \longrightarrow 0 \quad \text { as } j \longrightarrow \infty .
$$
This concludes the proof.
\end{proof}

\section{Approximation by smooth and compactly supported functions}\label{sec3}
 This section is aimed at proving  some density properties of smooth and compactly supported functions in the new fractional Musielak-Sobolev spaces.
\begin{theorem}\label{theo3.1}
Assume that \eqref{eq1.333} and \eqref{eq1.5} hold. Let $u \in W^{s,\Phi_{x,y}}\left(\mathbb{R}^N\right)$, then for any fixed $\delta>0$, there exists a continuous and compactly supported function $u_\delta$ such that
$$
\left\|u-u_\delta\right\|_{W^{s,\Phi_{x,y}}\left(\mathbb{R}^N\right)} \longrightarrow 0 \quad \text { as } \delta \longrightarrow 0 .
$$
\end{theorem}
\begin{proof}
Let $u \in W^{s,\Phi_{x,y}}\left(\mathbb{R}^N\right)$, then $u\in L_{\widehat{\Phi}_{x}}(\mathbb{R}^{N}) $. Thus according to Lemma \ref{2lem2}, we can assume that $u \in L^{\infty}\left(\mathbb{R}^N\right)$.\\
Let $\tau_j \in C^{\infty}\left(\mathbb{R}^N\right)$ be as in Section \ref{sec2}, with $\tau_j(P)=1$ if $|P| \leq j$ and $\tau_j(P)=0$ if $|P| \geq j+1$. Let $u_j=\tau_j u$. Then
$$
u_j \longrightarrow u \text { a.e. in } \mathbb{R}^N \text { as } j \longrightarrow \infty
$$
and
$$
\widehat{\Phi}_{x}\left(\left|u(x)-u_j(x)\right|\right) \leq 2^{\varphi^{+}}\widehat{\Phi}_{x}\left(|u(x)|\right) \in L^1\left(\mathbb{R}^N\right) .
$$
Therefore, by using the dominated convergence theorem, we get
$$
\int_{\mathbb{R}^N}\widehat{\Phi}_{x}\left(\left|u(x)-u_j(x)\right|\right) d x \longrightarrow 0 \quad \text { as } j \longrightarrow \infty .
$$
Hence, for any fixed $\delta>0$, there exists $j_\delta \in \mathbb{N}$ such that
\begin{equation}\label{eqdelta1}
\int_{\mathbb{R}^N}\widehat{\Phi}_{x}\left(\left|u(x)-u_{j_\delta}(x)\right|\right) d x \leq \delta .
\end{equation}
Since $u_{j_\delta}$ is supported in $\overline{B}_{j+1}$ and $\displaystyle\mu(A)=\int_A d x$ is finite over compact sets, then by using Lusin's theorem (see \cite[Theorem~7.10]{folland1984real} for the definition of the uniform norm), we obtain that there exist a closed set $E_\delta \subset$ $\mathbb{R}^N$ and a continuous and compactly supported function $u_\delta: \mathbb{R}^N \longrightarrow \mathbb{R}$ such that
$$
u_\delta=u_{j_\delta} \quad \text { in } \mathbb{R}^N \backslash E_\delta, \quad \mu\left(E_\delta\right) \leq \delta \quad \text { and } \quad\left\|u_{\delta}\right\|_{L^{\infty}\left(\mathbb{R}^N\right)} \leq\|u_{j_\delta}\|_{L^{\infty}\left(\mathbb{R}^N\right)}.
$$
In particular, since $0 \leq \tau_{j_\delta}(x) \leq 1$, we have
$$
\left\|u_\delta\right\|_{L^{\infty}\left(\mathbb{R}^N\right)} \leq\|u\|_{L^{\infty}\left(\mathbb{R}^N\right)}<\infty .
$$
Therefore, by Lemma \ref{lemmapro} and the assumption \eqref{eq1.333}, we get
$$
\begin{aligned}
\int_{\mathbb{R}^N}\widehat{\Phi}_{x}\left(\left|u_{j_\delta}(x)-u_\delta(x)\right|\right) d x & =\int_{E_\delta}\widehat{\Phi}_{x}\left(\left|u_{j_\delta(x)}-u_\delta(x)\right|\right) d x \\
&\leq 2^{\varphi^{+}-1}\left(\int_{E_\delta} \widehat{\Phi}_{x}\left(| u_{j_{\delta}}(x)|\right) d x +\int_{E_\delta}\widehat{\Phi}_{x}\left(|u_\delta(x)|\right) dx \right) \\
& \leq C\left(\varphi^{+}, \varphi^{-},\|u\|_{L^{\infty}(\mathbb{R}^N)}\right)  \mu\left(E_\delta\right).
\end{aligned}
$$
Hence
\begin{equation}\label{eqdelta2}
\int_{\mathbb{R}^N}\widehat{\Phi}_{x}\left(\left|u_{j_\delta}(x)-u_\delta(x)\right|\right) d x \leq C \delta .
\end{equation}
On the other hand, we have
$$
\begin{aligned}
&\widehat{\Phi}_{x}\left(| u(x) -u_\delta(x)|\right) \\
& \leq 2^{\varphi^{+}-1}\left( \widehat{\Phi}_{x}\left(|u(x)-u_{j_\delta}(x)|\right)+\widehat{\Phi}_{x}\left(|u_{j_\delta}(x)-u_\delta(x)|\right) \right) .
\end{aligned}
$$
Hence
$$
\begin{aligned}
\int_{\mathbb{R}^N}\widehat{\Phi}_{x}\left(|u(x)-u_\delta(x)|\right) d x \leq & 2^{\varphi^{+}-1} \int_{\mathbb{R}^N}\widehat{\Phi}_{x}\left(|u(x)-u_{j_\delta}(x)|\right) d x \\
& +2^{\varphi^{+}-1} \int_{\mathbb{R}^N}\widehat{\Phi}_{x}\left(|u_{j_\delta}(x)-u_\delta(x)|\right) d x .
\end{aligned}
$$
Then, from \eqref{eqdelta1} and \eqref{eqdelta2}, we deduce that
$$
\int_{\mathbb{R}^N}\widehat{\Phi}_{x}\left(|u(x)-u_\delta(x)|\right) d x \longrightarrow 0 \quad \text { as } \delta \longrightarrow 0 .
$$
Therefore, by Proposition \ref{2prop1}, we obtain
$$
\left\|u-u_\delta\right\|_{L_{\widehat{\Phi}_{x}}\left(\mathbb{R}^N\right)} \longrightarrow 0 \text { as } \delta \longrightarrow 0 .
$$
Moreover, we have
$$
\int_{\mathbb{R}^N} \int_{\mathbb{R}^N} \Phi_{x,y}\left(|\left(u-u_j\right)(x)-\left(u-u_j\right)(y)|K(x,y)\right)d x d y \longrightarrow 0 \quad \text { as } j \longrightarrow \infty .
$$
Hence, for any fixed $\delta>0$, there exists $j_\delta \in \mathbb{N}$ such that
\begin{equation}\label{ineq1}
\int_{\mathbb{R}^N} \int_{\mathbb{R}^N} \Phi_{x,y}\left(|\left(u-u_{j_\delta}\right)(x)-\left(u-u_{j_\delta}\right)(y)|K(x,y)\right)d x d y \leq \delta .
\end{equation}
Notice that
$$
v_{j_\delta}(x, y)=\left|u_{j_\delta}(x)-u_{j_\delta}(y)\right|K(x,y)
$$
is supported in $\left\{P \in \mathbb{R}^N \times \mathbb{R}^N ;|P| \leq j_\delta+1\right\}$ and $\displaystyle\mu(A)=\int_A \int_A d x d y$ is finite over compact sets. Therefore, by using Lusin's theorem, we get that there exist a closed set $E_\delta \subset \mathbb{R}^N \times \mathbb{R}^N$ and a continuous and compactly supported function $u_\delta: \mathbb{R}^N \times \mathbb{R}^N \longrightarrow \mathbb{R}$ such that
$$
\begin{aligned}
u_\delta & =v_{j_\delta} \quad \text { in } \mathbb{R}^N \times \mathbb{R}^N \backslash E_\delta, \quad \mu\left(E_\delta\right) \leq \delta \quad \text { and } \\
\left\|u_{\delta}\right\|_{L^{\infty}\left(\mathbb{R}^N \times \mathbb{R}^N\right)} & \leq\left\|v_{j_{\delta}}\right\|_{L^{\infty}\left(\mathbb{R}^N \times \mathbb{R}^N\right)} .
\end{aligned}
$$
In particular, since $0 \leq \tau_{j_{\delta}}(x) \leq 1$, we have
$$
\left\|u_\delta\right\|_{L^{\infty}\left(\mathbb{R}^N \times \mathbb{R}^N\right)} \leq \left\|| u(x)-u(y)|K(x,y) \right\|_{L^{\infty}\left(\mathbb{R}^N \times \mathbb{R}^N\right)} .
$$
Now, by putting
$$
v(x, y):=|u(x)-u(y)|K(x,y)
$$
in Lemma \ref{2lem3}, we can assume that
$$
|u(x)-u(y)|K(x,y) \in L^{\infty}\left(\mathbb{R}^N \times \mathbb{R}^N\right) .
$$
Therefor
$$
\int_{\mathbb{R}^N} \int_{\mathbb{R}^N} \Phi_{x,y}\left(\left|\left(u-u_\delta\right)(x)-\left(u-u_\delta\right)(y)\right|K(x,y)\right)d x d y \longrightarrow 0 \quad \text { as } \delta \longrightarrow 0,
$$
which concludes the proof.
\end{proof}

\begin{theorem}\label{theo3.2}
We assume that \eqref{eq1.3}, \eqref{eq1.4} and \eqref{eq1.5} hold. Then the space $C_0^{\infty}\left(\mathbb{R}^N\right)$ is dense in $W^{s,\Phi_{x,y}}\left(\mathbb{R}^N\right)$.
\end{theorem}
\begin{proof}
 We show that for any $u \in W^{s,\Phi_{x,y}}\left(\mathbb{R}^N\right)$, there exists a sequence $\rho_{\varepsilon} \in$ $C_0^{\infty}\left(\mathbb{R}^N\right)$ such that
$$
\left\|\rho_{\varepsilon}-u\right\|_{W^{s,\Phi_{x,y}}\left(\mathbb{R}^N\right)} \longrightarrow 0 \quad \text { as } \varepsilon \longrightarrow 0 .
$$
Let $u \in W^{s,\Phi_{x,y}}\left(\mathbb{R}^N\right)$, and let us fix $\delta>0$. Let $\tau_j$ be as in Section \ref{sec2}. From Lemma \ref{lem2.8}, we have
$$
\left\|u-\tau_j u\right\|_{W^{s,\Phi_{x,y}}\left(\mathbb{R}^N\right)} \leq \frac{\delta}{2},
$$
for $j$ large enough.
For any $\varepsilon>0$, let us consider
$$
\rho_{\varepsilon}:=\tau_j u * J_{\varepsilon},
$$
where $J_{\varepsilon}$ is the mollifier function defined in Section \ref{sec2}. By construction, $\rho_{\varepsilon} \in$ $C^{\infty}\left(\mathbb{R}^N\right)$. Moreover, by \cite[Proposition~IV.18]{haim1983analyse}, we have
$$
\operatorname{supp} \rho_{\varepsilon} \subseteq \operatorname{supp}\left(\tau_j u\right)+\overline{B}_{\varepsilon} .
$$
Also, by Lemma \ref{lem2.8}, we have
$$
\operatorname{supp}\left(\tau_j u\right) \subseteq \overline{B}_{j+1} \cap \operatorname{supp} u .
$$
Thus,
$$
\operatorname{supp} \rho_{\varepsilon} \subseteq\left(\overline{B}_{j+1} \cap \operatorname{supp} u\right)+\overline{B}_{\varepsilon}.
$$
Hence,
$$
\rho_{\varepsilon} \in C_0^{\infty}\left(\mathbb{R}^N\right),
$$
for $\varepsilon$ small enough. Furthermore, by Lemma \ref{lem2.06}, we have
$$
\left\|\rho_{\varepsilon}-\tau_j u\right\|_{W^{s,\Phi_{x,y}}\left(\mathbb{R}^N\right)} \leq \frac{\delta}{2}
$$
for $\varepsilon$ small enough. Therefore,
$$
\begin{aligned}
\left\|u-\rho_{\varepsilon}\right\|_{W^{s,\Phi_{x,y}}\left(\mathbb{R}^N\right)} & \leq\left\|u-\tau_j u\right\|_{W^{s,\Phi_{x,y}}\left(\mathbb{R}^N\right)}+\left\|\tau_j u-\rho_{\varepsilon}\right\|_{W^{s,\Phi_{x,y}}\left(\mathbb{R}^N\right)} \\
& \leq \frac{\delta}{2}+\frac{\delta}{2} \\
& \leq \delta
\end{aligned}
$$
Since $\delta$ can be taken arbitrarily small, then the proof follows.
\end{proof}

\begin{theorem}\label{theo3.3}
We assume that \eqref{eq1.3}, \eqref{eq1.4} and \eqref{eq1.5} hold and suppose that $\Omega$ is a $W^{s,\Phi_{x,y}}$-extension domain. Then $C^{\infty}(\overline{\Omega})$ is dense in $W^{s,\Phi_{x,y}}(\Omega)$.
\end{theorem}

\begin{proof}
Let $u \in W^{s,\Phi_{x,y}}(\Omega)$. Since $\Omega$ is a $W^{s,\Phi_{x,y}}$-extension domain, then there exist $\widetilde{u} \in W^{s,\Phi_{x,y}}(\mathbb{R}^{N})$ such that $\widetilde{u}(x)=u(x)$ for all $x \in \Omega$ and
$$
\|\widetilde{u}\|_{W^{s,\Phi_{x,y}}\left(\mathbb{R}^N\right)} \leq C\|u\|_{W^{s,\Phi_{x,y}}(\Omega)} .
$$
Thus, according to Theorem \ref{theo3.2}, we can choose $\widetilde{u}_{\varepsilon} \in C_0^{\infty}\left(\mathbb{R}^N\right)$ with
$$
\widetilde{u}_{\varepsilon} \longrightarrow \widetilde{u} \quad \text { in } W^{s,\Phi_{x,y}}\left(\mathbb{R}^N\right) .
$$
We set $u_{\varepsilon}:=\left.\widetilde{u}_{\varepsilon}\right|_{\Omega}$. Thus,
$$
\left\|u-u_{\varepsilon}\right\|_{W^{s,\Phi_{x,y}}(\Omega)} \leq\left\|\widetilde{u}-\widetilde{u}_{\varepsilon}\right\|_{W^{s,\Phi_{x,y}}\left(\mathbb{R}^N\right)} \longrightarrow 0 .
$$
Hence $u_{\varepsilon} \in C^{\infty}(\overline{\Omega})$ are the required approximating functions.
\end{proof}

\bibliographystyle{plain}

\end{document}